\documentclass[11pt]{amsart}

\usepackage[latin1]{inputenc}
\usepackage{graphicx}
\usepackage{amsmath}
\usepackage{amssymb}

\usepackage{amsmath,amsthm,amsfonts,amssymb,amscd}
\usepackage{enumerate}% http://ctan.org/pkg/enumerate
\usepackage{amssymb,amsthm,amsmath,psfrag,mathrsfs}
\usepackage[latin1]{inputenc}
\usepackage{graphicx}

\newtheorem{theorem}{Theorem}

\newtheorem{Theorem}{Theorem}[section]
\newtheorem{Corollary}{Corollary}[section]
\newtheorem{Proposition}{Proposition}[section]
\newtheorem{Lemma}{Lemma}[section]
\newtheorem{Claim}{Claim}[section]
\theoremstyle{Definition}
\newtheorem{Definition}{Definition}[section]
\newtheorem{Example}{Example}[section]

\theoremstyle{Remark}
\newtheorem{Remark}{Remark}[section]

\def\leaderfill{\leaders\hbox to .8em{\hss .\hss}\hfill}
\def\_#1{{\lower 0.7ex\hbox{}}_{#1}}

\def\fa{{\mathcal{F}}}

\def\vr{{\varphi}}

\def\Aut{\operatorname{{Aut}}}

\def\sing{\operatorname{{sing}}}

\def\cod{\operatorname{{cod}}}

\def\deg{\operatorname{{deg}}}

\title[Integrable deformations of foliations]{Integrable deformations of foliations:
 a generalization of Ilyashenko's result}
\author{Dominique Cerveau and Bruno Scárdua}

\address{Dominique Cerveau: Universit\'e de Rennes /CNRS
- IRMAR- UMR 6625, F 35000 - Rennes, France}
\email{dominique.cerveau@univ-rennes1.fr}

\address{Bruno Sc\'ardua: Inst. Matem\'atica,  Universidade Federal do Rio %de Janeiro.   Caixa Postal 68530,
Rio de Janeiro-RJ,  21.945-970 BRAZIL} \email{scardua@im.ufrj.br}

\begin{document}

\maketitle

\begin{abstract}
We study analytic deformations of holomorphic differential 1-forms.
The initial 1-form is exact homogeneous and the deformation is by
polynomial integrable 1-forms. We investigate under which conditions
the elements of the deformation are still exact or, more generally,
exhibit a first integral.  Our results are related to natural
extensions of classical results of Ilyashenko on limit cycles of
perturbations of hamiltonian systems in two complex variables.
\end{abstract}
\tableofcontents

\section{Introduction and main results}
\label{section:introduction} In the year of 1969 Ilyashenko
published his PhD thesis about limit cycles of two-dimensional
 analytic ordinary differential equations (\cite{Ilyashenko}). Working
with perturbations of a hamiltonian of a complex polynomial in the
complex affine space $\mathbb C^2$. Ilyashenko's result can be
summarized as follows:

let $R(z,w)$ be a degree $n+1$ complex polynomial and $A(z,w), \,
B(z,w)$ complex polynomials of degree $n$, $t$ a complex parameter
and consider the following perturbation

\begin{equation}
\label{equation:1} \frac{dw}{dz}=-\frac{ R_z + t A}{ R_w + tB}
\end{equation}

of the hamiltonian equation

\begin{equation}
\label{equation:2} \frac{dw}{dz}= - \frac{R_z}{R_w}
\end{equation}
Let $B_{n+1}$ denote the space of coefficients of polynomials of
degree $n+1$ in $\mathbb C[x,y]$ and let $B_{n+1} ^{\prime
\prime}\subset B_{n+1}$ denote the space of those polynomials $R$
for which equation~\eqref{equation:2} has exactly $n^2$ singular
points. Finally, denote by $B_{n+1}^\prime \subset B_{n+1} ^{\prime
\prime}$ the space of those polynomials $R \in B_{n+1}^\prime$ for
which the singular points of \eqref{equation:2} lie on distinct
level sets of $R$. Then $B_{n+1}^\prime$ is a Zariski (and therefore
dense) open set in $B_{n+1}$.

Recall that a singularity of a real vector field in the real plane
is a {\it center} if it admits a neighborhood where all the
non-singular orbits are closed. In the complex world according to
\cite{Ilyashenko} a singular point of the equation (2) is said to be
a {\it center} if the foliation of its neighborhood into solutions
is topologically equivalent to the foliation of a neighborhood of
$(0,0)\in \mathbb C^2$ by curves $zw = const$ (or equivalently $z^2
+ w^2 = const$). In view of Mattei-Moussu theorem
\cite{mattei-moussu} the topological equivalence above can be
assumed to be holomorphic.

If we denote by $\Omega_n$ the space of 1-forms $\omega= A dz + B
dw$ where $A,B \in B_n$  then according to Ilyashenko
(\cite{Ilyashenko} Corollary 1) we have:

\begin{Theorem}
\label{theorem:A} Let $R \in B_{n+1} ^\prime$ and let $P$ be a
singular point of \eqref{equation:2} for which $A(P)=B(P)=0$ (i.e.,
$P$ is also a singular point of \eqref{equation:1}). If $P$ is also
a center for \eqref{equation:1} for all $t \approx 0$ then $\omega =
A dz + B dw \in \Omega_n$ is exact and in particular
\eqref{equation:1} admits a polynomial first integral.

\end{Theorem}

This important result is based on the following
integration lemma
(cf. \cite{Ilyashenko} Theorem 1):

\begin{Theorem}[Ilyashenko's integration lemma]
\label{theorem:B} Let $R \in B_{n+1} ^{\prime \prime}$ and $\omega
\in \Omega_n$. Then $\omega$ is exact $\omega=dS$ for some $S \in
B_{n+1}$ if, and only if $\int\limits_{\gamma}\omega=0$ for all
closed curve $\gamma \subset \{ R=c\}, \forall c \in \mathbb C$.
\end{Theorem}
In short, {\it $\omega\in \Omega_n$ is exact provided that its
restrictions to the fibers of $R \in B_{n+1}^{\prime \prime}$ are
exact.}

This paper can be seen as a natural extension of these results to
higher dimension, for the case of codimension one foliations.
Indeed, let us rewrite (2) as $R_z dz + Rw dw=0$ and (1) as $(R_z +
tA)dz + (R_w + tB)dw=0$.  Then we put $\omega_t:= dR + t(A dz + B
dw)$ and $\omega_0=dR$, so that $\omega_t$ is an analytic
deformation of $\omega_0=dR$. Also we have
equation~\eqref{equation:1} is equivalent to $\omega_0=0$ and
equation~\eqref{equation:2}  is equivalent to $\omega_t=0$.

The deformation writes $\omega_t = \omega_0 + t \omega_1$ where
$\omega_1:= Adz + B dw$. This is a {\it degree one} (in the
parameter $t$) deformation of $\omega_0=dR$ by 1-forms of degree
$n$. Using this point of view and notation we can state Ilyashenko's
result above as follows:

\begin{Theorem}[Ilyashenko]
\label{theorem:C} Let $\omega_t=dR + t \omega_1$ be a degree one
analytic deformation of the 1-form $\omega_0=dR$ where $R \in
B_{n+1}^{\prime \prime}$ and each $\omega_t$ is polynomial of degree
$n$. Then the following conditions are equivalent:

\begin{enumerate}[{\rm(i)}]
\item $\omega_t$ is also of hamiltonian type, i.e.,
$\omega_t = dR_t$ for some $R_t\in B_{n+1}^{\prime \prime}, \forall
t \approx 0$.

\item Given a singularity $P\in \mathbb C^2$ of $\omega_0=dR$ there
exists an analytic curve $P_t \colon \mathbb C,0 \to \mathbb C^2,0$
such that $P_0=P$ and $P_t$ is a center type singularity of
$\omega_t$ (the only one near $P$).

\item There is a singularity $P\in \mathbb C^2$ of $\omega_o=dR$ and
there is an analytic curve $P_t \colon \mathbb C,0 \to \mathbb C^2,
0$ such that $P_0=P$ and $P_t$ is a center type singularity of
$\omega_t$.

\end{enumerate}
\end{Theorem}

In this paper we investigate then analytic deformations of
foliations admitting a first integral of polynomial homogeneous
type. The deformations are required to be given by integrable
1-forms following the integrability condition, since we work in
dimension $n \geq 2$. The main object of our study is then an
analytic family of holomorphic 1-forms $\{\omega_t\}_{t \in \mathbb
C,0}$ where each $\omega_t$ is a holomorphic 1-form (mostly
polynomial) on a neighborhood $U$ of the origin $0\in \mathbb C^n, n
\geq 3$. We assume that each $\omega_t$ is integrable, i.e.,
$\omega_t \wedge d \omega_t=0$ so that $\omega_t$ defines a
codimension one holomorphic foliation off its singular set
$\sing\omega_t=\{p \in U, \, \omega_t(p)=0\}$. We write $\omega_t =
\omega_0 + \sum\limits_{j=1}^\infty t^j\omega_j$ where $\omega_0$ is
an integrable 1-form. The deformation is called {\it degree one
deformation} when $\omega_t = \omega_0 + t \omega_1$. We shall
consider the case where $\omega_0$ admits a first integral
$\omega_0=df$, which is assumed to be polynomial or holomorphic in a
neighborhood of the origin. Another possibility is to  investigate
the case where $\omega_0$ is logarithmic $\omega_0 = (f_1\ldots
f_{r+1})\sum\limits_{j=1} ^{r+1} \lambda _j df_j / f_j$ where the
$f_j$ are like $f$ above and $\lambda_j \in \mathbb C$. This case
will be considered in a forthcoming work.

\subsection{Singular Frobenius, cycles and persistence of first integrals}

According to a well-known theorem of Malgrange (\cite{malgrangeI}) a
germ of holomorphic 1-form $\omega$ satisfying the integrability
condition $\omega\wedge d \omega=0$ admits a germ of holomorphic
first integral, at a singular point where the singular set has
codimension $\geq 3$. Since this condition in the singular set is
stable under small deformations, we conclude that any small
deformation of such object admits a holomorphic first integral. Our
Theorem~\ref{Theorem:A} below  gives an extension of this last
conclusion for the case where the singular set has codimension $\geq
2$, but with normal crossings condition. Let us be more precise. The
center persistence condition in Ilyashenko's results is equivalent
in our framework to the vanishing of some line integrals associated
to the deformation. Such conditions are automatically satisfied
under some non-degeneracy and irreducibility hypotheses on the first
integral of $\omega_0$. Our conditions look like the vanishing of
the Melnikov functions of the deformation (\cite{movasati} page 11)
and they are intrinsic.

Let us state our main results. Our first result is strongly related
to Ilyashenko's theorem (Theorem~\ref{theorem:A}).
\begin{theorem}
\label{Theorem:A}
 Let $f=f_1\ldots f_{r+1}$ be a product of
irreducible homogeneous polynomials $f_j \in \mathbb C[x_1,...,x_n]$
with $<f_i,f_j>=1$ for $i  \ne j$. Assume that the corresponding
germ induced by $f$ at the origin, has only normal crossings
singularities except for a codimension $\geq 3$ analytic subset. Let
$\omega_t = df + \sum\limits_{j=1}^\infty t^j \omega_j$ be an
analytic deformation of $\omega_0=df$ such that:

\begin{enumerate}[{\rm(i)}]
\item Each $\omega_t$ is a polynomial integrable 1-form.

\item We have  $\deg \omega_t \leq \deg \omega_0, \forall t$.

\end{enumerate}

Then the following conditions are equivalent:

\begin{enumerate}[{\rm(a)}]

\item $\omega_t$ is exact for each $t\approx 0$.
\item  There is a map $F_t \colon t \mapsto \mathbb C[x_1,...,x_n]$ such  that
$F_t(x_1,..,x_n)$ is a degree $\deg(F_t)\leq \deg \omega_0+1$
polynomial satisfying $dF_t = \omega_t$ for each $t \approx 0$.

\item We have $\oint_{\gamma_{c ^{(j)}}}\omega_t =0$ for a set of
generators $\{\gamma_c ^{(j)}, \, j=1,...,r\}$ of the 1-homology of
the leaf $L_c: (f=c), \forall c \ne 0$.
\item We have $\oint_{\gamma_{c ^{(j)}}}\omega_t/f =0$ for a set of
generators $\{\gamma_c ^{(j)}, \, j=1,...,r\}$ of the 1-homology of
the leaf $L_c: (f=c), \forall c \ne 0$.

\end{enumerate}

\end{theorem}

\begin{Remark}{\rm
First we remark that in the course of the proof of
Theorem~\ref{Theorem:A} it will be recursively established that the
1-forms $\omega_t$ are closed in the fibers of $f$. Indeed, it will
be first observed that, by the integrability condition, $\omega_1$
is closed in the fibers of $f$. Then, the condition
$\oint_{\gamma_{c ^{(j)}}}\omega_t/f=0$ will be used to prove that
$\omega_2$ is closed in the fibers of $f$ and so on.  Therefore,
there is no ambiguity in the  above integrals $\oint_{\gamma_{c
^{(j)}}}\omega_t$ and $\oint_{\gamma_{c ^{(j)}}}\omega_t/f$.

Let $f=f_1\ldots f_{r+1}\in \mathbb C[x_1,...,x_n], n \geq 2$ be as
in Theorem~\ref{Theorem:A}. We shall see (cf.
Lemma~\ref{Lemma:1homology}) that there are generators
$\{\theta_1,...,\theta_r\}$ of the 1-homology of the non-singular
fibers $L_c : (f=c), c \ne 0$ which are of the form $\theta_j=
\sum\limits_{k=1}^{r+1} \lambda_k ^{(j)} df_k / f_k$ for a suitable
choice of the coefficients $\lambda_k ^{(j)}\in \mathbb C$.

The integral condition cannot be dropped in Theorem~\ref{Theorem:A}
(cf. Example~\ref{Example:cycle}). Some additional remarks about
Theorem~\ref{Theorem:A} are:
\begin{enumerate}
\item If $\omega_t$ is homogeneous for all $t$ then the first
integral is homogeneous as well.

\item We actually prove that each $\omega_t$ is exact which is a
sort of generalization of Ilyashenko's integration lemma
(Theorem~\ref{theorem:B}).

\item If $f=f_1$ is irreducible, reduced  and has only normal crossings singularities
 off a codimension $\geq 3$ analytic subset, then the non-singular fibers
$L_c: (f=c), c \ne 0, c \approx 0$ are simply-connected (cf.
Lê-Saito's theorem in \cite{Le-Saito} or else
Theorem~\ref{Theorem:LeSaito}). Therefore  the integral condition
$\oint_{\gamma_c}\omega_t=0$ is automatically verified.

\end{enumerate}

More precisely we have:

\begin{Corollary}
\label{Corollary:1} Let $P\in \mathbb C[x_1,...,x_n]$ be a
homogeneous polynomial, $n \geq 3$. Assume that $P$ is irreducible
and $X_P: (P=0)\subset \mathbb C^n$ has only normal crossings type
singularities outside of a codimension $\geq 3$ subset in a
neighborhood of the origin  of $\mathbb C^n$. Then any analytic
deformation $\omega_t$ of $\omega_0=dP$ by polynomial integrable
1-forms, of degree $\deg(\omega_t) \leq \deg(dP)$, also exhibits
polynomial first integrals for $t$ close to $0$. Indeed, there is an
analytic family of polynomials $P_t\in \mathbb C[x_1,...,x_n]$ of
degree $\deg(P_t)\leq \deg(P)$ such that $P_0=P$ and $\omega_t=dP_t,
\forall t \approx 0$.

\end{Corollary}

} \end{Remark}

The condition $\deg(\omega_t) \leq \deg(\omega_0)$ cannot be
dropped. Indeed, consider $\omega_t = dy + ty dx$ on $\mathbb C^2$.
Then $\omega_0=dy$ but $\omega_t$ is not closed for $ t\ne 0$. The
point is that the degree of $\omega_t$ is $1$ for each $t \ne 0$,
while $\omega_0$ has degree $0$. Moreover, the family of examples
$\omega_t=d(xyz) +t(xyz) (adx/x + bdy/y +cdz/z), a, b, c \in \mathbb
C$ shows that the irreducibility hypothesis on $P=P_{\nu+1}$ cannot
be dropped.

Another interesting application of our techniques is the following
result. It has already been proved in \cite{Cerveau-Mattei} with
more geometrical arguments, based on Deligne's theorem
(\cite{deligne}), holonomy arguments (\cite{mattei-moussu}) and some
desingularization techniques (\cite{C-LN-S1}). Here we present a
proof using our techniques of deformation.

\begin{theorem}
\label{Theorem:B} Let $\Omega$ be a germ of  integrable holomorphic
1-form at the origin $0 \in \mathbb C^n, n \geq 3$. Assume that the
first jet of $\Omega$  is of the form $\Omega_\nu=dP_{\nu+1}$ for
some irreducible homogeneous polynomial $P_{\nu+1}\in \mathbb
C[x_1,...,x_n]_{\nu+1}$ having only normal crossings singularities
outside of a codimension $\geq 3$ subset. Then $\Omega$ also admits
a holomorphic first integral in a neighborhood of the origin.

\end{theorem}

Notice that Theorem~\ref{Theorem:B} really requires $P=P_{\nu+1}$ to
be irreducible as shown by the family of examples $\omega_t=d(xyz)
+t(xyz)^n (adx/x + bdy/y +cdz/z), a, b, c \in \mathbb C, \, n \in
\mathbb N$. Indeed, for a generic choice of the coefficients $a, b,
c$ there are no holomorphic first integrals.

\vglue.1in

Before stating our next result we recall a classical result due to
G. Reeb \cite{Reeb} (see also \cite{C-LN} page 85):
\begin{Theorem}[Reeb,\cite{Reeb}]
\label{Theorem:ReebD} Let $\omega$ be an analytic integrable 1-form
defined in a neighborhood of the origin $ 0 \in \mathbb R^n, n \geq
3$. Suppose that $\omega(0)=0$ and $\omega$ has a non-degenerate
linear part $\omega_1 = df$, i.e., $f$ is a quadratic form of
maximal rank (not necessarily of center type). Then there exist an
analytic diffeomorphism $h \colon (\mathbb R^n,0) \to (\mathbb
R^n,0)$ and an analytic function $g \colon (\mathbb R^n,0) \to
(\mathbb R,0)$ with $h^*(\omega)=gdf$.
\end{Theorem}

 We stress that the singularity is not necessarily of center type.

The above theorem has some version for $\omega$ of class $C^2$ but
demanding that the singularity is of center type.

In our case we shall consider some versions of Reeb's theorem above.
We shall work with holomorphic integrable 1-forms of type $\Omega=
dP + \Omega^\prime$ where $P$ is a homogeneous irreducible
polynomial, and $\Omega^\prime$ is a 1-form of higher order terms
than $dP$. Under some hypotheses on $P$ we shall conclude that also
$\Omega$ admits a first integral which is a perturbation of $P$.
This includes for instance the case $P=\sum\limits_{j=1}^n x_j ^d,
\, n \geq 3, d \geq 2$ a so called {\it Pham polynomial}. Given a
polynomial $P \in \mathbb R[x_1,...,x_n]$ we denote by $P^{\mathbb
C} \in \mathbb C[z_1,...,z_n]$ its complexification where $z_j= x_j
+ \sqrt{-1} y_j$. As for the real analytic case we can state:

\begin{Corollary}
Let $\omega$ be an analytic integrable 1-form defined in a
neighborhood of the origin $ 0 \in \mathbb R^n, n \geq 3$. Suppose
that $\omega(0)=0$ and $\omega$ has a first jet of the form
 $\omega_\nu = dP_{\nu+1}$ where $P_{\nu+1}$ is a homogeneous polynomial
 of degree $\nu + 1 \geq 2$. Assume that:
 \begin{enumerate}
\item $P^{\mathbb C} _{\nu+1}=0$ has only normal type singularities except for a
codimension $\geq 3$ subset;
\item $P^{\mathbb C} _{\nu+1}$ is irreducible in $\mathbb C[x_1,...,x_n]$.
\end{enumerate}
Then $\omega$ also admits an analytic first integral $f \colon
(\mathbb R^n,0) \to (\mathbb R,0)$ which is a perturbation of
$P_{\nu+1}$.
\end{Corollary}

The proof of this corollary goes as one imagines: complexification
of $\omega$, Theorem~\ref{Theorem:B} and then back to the real
framework.

The above corollary then gives a new (?) proof of Reebs's
linearization theorem mentioned above.

\subsection*{Homogeneous deformations {\rm(cf.}\,\cite{Cerveau-Mattei}{\rm)}}
 Let us now give a word about the case
of deformations by {\it homogeneous integrable 1-forms} of a 1-form
$\omega_0=df$ where $f$ is a homogeneous polynomial. In this case,
using the description of non-dicritical homogeneous integrable
1-forms given in \cite{Cerveau-Mattei} part 4, Chap. I pp 86-95 we
can rapidly describe such deformations. For instance, if it is
required that the set of separatrices $(f=0)$ is left invariant
during the deformation then such deformations are then proved to be
of logarithmic type. It is important to notice that by the use of
\cite{Cerveau-Mattei}, {\em no additional hypotheses are made on the
singular locus of $df$}. The main point is the fact that given a
homogeneous integrable 1-form $\omega$ then either $\omega(R )=0$,
where $R$ is the radial vector field, or $\omega/\omega(R )$ is a
closed homogeneous 1-form of degree $-1$. Then the description of
closed meromorphic 1-forms (cf. prop. 2.2 page 39 in
\cite{Cerveau-Mattei}) finishes the job. We may then derive  from
\cite{Cerveau-Mattei} the following conclusion:

\begin{Theorem} [cf. \cite{Cerveau-Mattei} Chap. 4.I]
\label{Theorem:3} Let $f=f_1\ldots f_{r+1}$ be a reduced product of
irreducible homogeneous polynomials in $\mathbb C^n , n \geq 2$. Let
$\omega_t= df +\sum\limits_{j=1} ^\infty t^j \omega_j$ be an
analytic deformation of $df$ by homogeneous integrable 1-forms of
same degree than $\omega_0=df$. Assume also that $(f=0)$ is
invariant for each $\omega_t$. Then $\omega_t$ is of logarithmic
type in the following sense:
\[
\omega_t/f= df/f + t \sum\limits_{k=1} ^{r+1} \lambda_k(t) df_k /
f_k
\]
for some $\lambda _k \in \mathcal O_1$.
\end{Theorem}

There is indeed a more general statement where we do not require
$(f=0)$ to be invariant. Let us start with a simple remark. Write
$\omega_t= df + \sum\limits_{j=1}^\infty t^j \omega_j$ where
$f=f_1\ldots f_{r+1}$. Then $\omega(R)=df(R) + \sum\limits_{j\geq 1}
t^j \omega_j(R) = (\nu +1)f + t \psi(t)$ for some holomorphic
function $\psi(t)$. Then for $ t \ne 0$ the polynomial $\omega_t(R)$
is reduced but may have less components than $\omega_0(R)=(\nu+1)f$.
This is the case for instance of the family of polynomials $g_t:=x^2
+ y^2 + t z^2$. For $t \ne 0$ we know that $g_t$ is irreducible. On
the other hand $g_0$ has two irreducible components. Taking this
into account we may then state:

\begin{Theorem} [cf. \cite{Cerveau-Mattei} Chap. 4.I]
\label{Theorem:3'} Let $f=f_1\ldots f_{r+1}$ be a product of
irreducible homogeneous polynomials in $\mathbb C^n , n \geq 2$. Let
$\omega_t= df +\sum\limits_{j=1} ^\infty t^j \omega_j$ be an
analytic deformation of $df$ by homogeneous integrable 1-forms of
same degree than $\omega_0=df$. There is a holomorphic function
$\epsilon(t)$, with $\epsilon(0)=0$  such that for
$g(t)=\omega_{\epsilon(t)}(R)$:
\[
\omega_{\epsilon(t)}= g_t  \sum\limits_{j=1} ^{k} \lambda_j(t)
dg_{j,t}/ g_{j,t}
\]
for some $\lambda _j(t) \in \mathcal O_1$ where $g_t= g_{1,t}\ldots
g_{k,t}$.
\end{Theorem}

\subsection{Degree one polynomial deformations}
Finally  we give a first step in the  study of {\em degree one}
deformations of $\omega_0=df$ in the general case. We assume that
the deformation is given by polynomial 1-forms under the hypothesis
that the degree does not grow. No integral (vanishing cycle type)
condition is required.

We shall need a definition:
\begin{Definition}
{\rm A holomorphic function $f \colon \mathbb C^n, 0 \to \mathbb C^p
,0$ $f=(f_1,...,f_p)$ {\it satisfies the factorization property} if
for every holomorphic function $ h \colon \mathbb C^n,0 \to \mathbb
C,0$ such that $dh \wedge df_1\wedge ...\wedge df_p=0$ there is a
holomorphic function $\vr \colon \mathbb C^p,0 \to \mathbb C,0$ such
that $h = \vr \circ f$. }

\end{Definition}

We have the following condition assuring the factorization property
due to Malgrange, with the above notations:
\begin{Theorem}[\cite{malgrangeII}]
Assume that $\cod \sing (df_1 \wedge ...\wedge df_p) \geq 2$. Then
$f=(f_1,...,f_p)$ satisfies the factorization property.

\end{Theorem}

The above result is not related to whether the map $f$ has connected
fibers. Indeed, $f_1(z_1,z_2,...,z_n)=(z_1,z_1 z_2)$ has connected
fibers but does not satisfy the factorization property (take
$h=z_2$). Furthermore, the map $f_2(z_1,z_2,z_3)=(z_1,z_1 z_2 +
z_3^2)$ does not have connected fibers. Nevertheless, since $dz_1
\wedge d(z_1 z_2 + z_3 ^2) = z_1 dz_1 \wedge dz_2 + 2 z_3 dz_1
\wedge dz_3$ the map $f_2$ satisfies the conditions of Malgrange's
theorem above and therefore satisfies the factorization property. In
addition to the above examples we consider $f_3=(z_1 ^2 z_2 ^3,
z_3)$. This example does not satisfy Malgrange's condition, it has a
non-irreducible component, but still verifies the factorization
property as it is easily checked. The fibers of $f_3$ are connected.
The difference with respect to $f_1$ is the fact that $f_3$ is an
open map, while $f_1$ is not. With this notions and remarks we
state:

\begin{theorem}
\label{Theorem:C} Let $f=f_1f_2$ be a product of two irreducible
homogeneous polynomials $f_1,f_2 \in \mathbb C[x_1,...,x_n]$ with
$<f_1,f_2>=1$. Assume that: \begin{enumerate}
\item The corresponding germ induced by $f$ at the origin, has only normal
crossings singularities except for a codimension $\geq 3$ analytic
subset.
\item The map $(f_1,f_2)\colon \mathbb C^n \to \mathbb C^2$
satisfies the factorization property. \end{enumerate}

 Then any affine integrable
deformation $\omega_t = df + t \omega_1$ by polynomial 1-forms of
degree $\deg(\omega_t ) \leq \deg(df)$ is of one of the following
types:
\begin{enumerate}[{\rm(a)}]
\item $\omega_t  = d(f + th)$, for some polynomial $h $ of
degree $\deg(h)\leq \deg (f)$.

\item $\omega_t=\sigma^*(\alpha_t)$ where $\sigma=(f_1,f_2)$ and
$\alpha_t= (1 + t \mu)d(xy) + t[d\big(P(x) + Q(y)\big) + \lambda
ydx]$ for some $\mu, \lambda \in \mathbb C$ and polynomials $P(x)\in
\mathbb C[x], \, Q(y) \in \mathbb C[y]$ of degree $\deg P \leq \deg
f_2, \, \deg Q \leq \deg f_1$.
\end{enumerate}

\end{theorem}

\begin{Remark}{\rm Some remarks about this theorem are:
\begin{enumerate}[{\rm(1)}]
\item  One may  search for examples of the form
 $\omega_t = d(fg) + t \mu d(fg) + th dg$ what, a priori, seems to be
possible. Nevertheless, we have $d \omega_t = tdh \wedge dg$ and
then $\omega_t \wedge d \omega_t = (1+ \mu t) d(fg) \wedge dh \wedge
dg$. Therefore $\omega_t \wedge d \omega_t =0\Leftrightarrow d(fg)
\wedge dh \wedge dg=0\Leftrightarrow g df \wedge dh \wedge dg=0
\Leftrightarrow dh \wedge df \wedge dg=0$. This last condition
implies, in the case $(f,g)$ satisfies the factorization property
that  $h=H(f,g)$ for some two variables polynomial $H(x,y)$.

\item  A natural idea is to construct deformations of logarithmic type
i.e,   $\omega_t / f$ is logarithmic of the form $\omega_t / f =
\sum\limits_{j=1} ^{2} \mu_j(t) df_j /f_j$ for some holomorphic
functions $\mu_j(t)$ with $\mu_j(0)=1$. Since we are considering
degree one deformations we must have $\mu_j(t) = 1 + \mu_j t$ for
some $\mu_j \in \mathbb C$. These cases are contained in case (b) in
the statement. Indeed, let us  consider deformations of the form
$\omega_t = xy[(1 + \mu_1 t)\frac{dx}{x} + (1 + \mu_2
t)\frac{dy}{y}]$. This deformation can be rewritten as $\omega_t =
(1 + t \mu_1) d(xy) + t(\mu_2 - \mu_1) xdy$ as in case (b) in the
statement.

\item  The above result cannot be proved without the
restriction on the degree of the 1-forms $\omega_t$. This is clear
from the already introduced family of examples $\Omega=d(xyz)
+t(xyz)^2 (adx/x + bdy/y +cdz/z), a, b, c \in \mathbb C$.

\item The above result  does not hold if we allow $f$ to be
non-reduced. For seeing this consider the example $\omega_t=d(x^2 y)
+ txy dy= 2xydx + (x^2 - t xy) dy$. Put $h_t = \omega_t \cdot R =
3x^2 y - txy^2$. Then $d(\omega_t/ h_t)=0$ where $\Omega_t:=
\omega_t / h_t = \frac{1}{3} dy/y + \frac{2}{3} d(3x  - ty)/((3x -
ty)$.

\end{enumerate}
} \end{Remark}

\section{Equations of a deformation}
\label{section:equations}

Let $\omega_t$ be  a {\it deformation} of $\omega_0$ a germ of
holomorphic 1-form at the origin $0 \in \mathbb C^n$, i.e.,
$\omega_t$ is a one-parameter analytic family of germs at the origin
$0 \in \mathbb C^n$ of holomorphic 1-forms parametrized by $t \in
\mathbb D\subset \mathbb C$. We shall assume that $\omega_t$ is
integrable for each $t$, i.e., $\omega _t \wedge d \omega_t=0, \,
\forall t\in \mathbb D$. We also write

\[
\omega_t = \omega_0 + \sum\limits_{j=1} ^\infty  t^j \omega_j\,.
\]

The integrability condition $\omega_t \wedge d \omega_t=0$ gives:

\[
\omega_0 \wedge d \omega_0=0
\]
\[
\omega_0 \wedge d\omega_1 + \omega_1 \wedge d\omega_0 =0
\]
\[
\omega_2 \wedge d \omega_0 + \omega_1 \wedge d \omega_1 + \omega_0 \wedge d\omega_2=0
\]
\[
\vdots
\]

We shall consider the case where  $\omega_0$ admits a first
integral, more precisely $\omega_0 = df$ for some holomorphic
function $f$. In this case
\[
df \wedge d \omega_1=0
\]
and
\[
\omega_1 \wedge d\omega_1 + df \wedge d \omega_2 =0
\]
\[
df\wedge d\omega_3+ \omega_1 \wedge d \omega_2 + \omega_2 \wedge d
\omega_1=0, ...
\]

These are called {\it equations of the deformation} in the case
where $\omega_0 = df$. Notice that $df \wedge d \omega_1$ means that
the 1-form $\omega_1$ is closed in the fibers of $f$ (\cite{Ce-Sc}
Lemma 5.1).

\begin{Example}{\rm [degree one deformations]
Suppose we have a {\it degree-one deformation} $\omega_t = \omega_0
+ t \omega_1$ where each $\omega_t$ is integrable. Since $\omega_0
=df$ we have the following equations for the deformation
\[
df \wedge d \omega_1=0, \, \omega_1 \wedge d \omega_1 =0
\]

} \end{Example}

\section{Local topology and homology of the fibers}

\label{section:Lesaito}

We consider $f\colon \mathbb C^n , 0 \to \mathbb C,0$ a germ of a
holomorphic function at the origin $0\in \mathbb C^n, n \geq 3$. The
corresponding germ of the analytic hypersurface $(f=0)$ is also
denoted by $X_f$.  The singular set of the hypersurface $X_f$ will
be denoted by $\sing(X_f)$.  Next we give a pleonastic definition of
our main hypothesis:

\begin{Definition}
{\rm We shall say that $X_f$ {\it has only ordinary singularities
off a codimension $\geq 3$ subset} if  there exists an analytic
subset $(Y,0)\subset (X_f,0)$ of dimension at most $n-3$, such that
outside of $Y$ the only singularities of $(X_f,0)$ are normal
crossings. }
\end{Definition}

We will assume that $f$ is {\it reduced} (if $g \in \mathcal O_n$ is
such that $g \big|_{X_f} \equiv 0$ then $f\big| g$ in $\mathcal
O_n$.). In this case the singular set of $X_f$ is given by
$\sing(X_f) = \sing(f) = \{p \in (\mathbb C^n,0): df(p)=0\}$.
Indeed, it is well-known (\cite{Milnor}) that the singular points of
$f$, i.e.,  the zeroes of $df$, are contained in the fiber
$f^{-1}(0)$. We consider the germ of an integrable 1-form $\omega
\in \Omega^1(\mathbb C^n, 0)$. Then $\omega =0$ defines a
codimension-one holomorphic foliation $\fa(\omega)$ germ at $0 \in
\mathbb C^n$. The hypersurface $X_f$ is $\fa(\omega)$-invariant if,
and only if, $\omega \wedge df/f$ is holomorphic. This is the case
of integrable 1-forms that write as
\[
\omega = a df + f \eta
\]
with $a \in \mathcal O_n$ and $\eta \in \Omega^1(\mathbb C^n,0)$.
  For $\eta$ small enough (in the sense of
Krull topology \cite{kaup-kaup}, \cite{gunning-rossi}) and $a \in
\mathcal O_n ^*$ unit, we may see $\fa(\omega)$ as an integrable
deformation of the holomorphic "fibration" $\fa(df)$, given by
$f=const.$. If for instance $f$ has an isolated singularity at $0
\in \mathbb C^n, n \geq 3$, then any $\omega$ that leaves $X_f :
(f=0)$ invariant must write as above, $\omega = adf + f \eta$
(\cite{Ce-Sc}). In particular, $\omega$ may come from an analytic
deformation of $\omega_0=df$, under some geometrical conditions as
in \cite{Ce-Sc}. In general however $a$ is not an unit. This somehow
explains the strength of the hypothesis $\omega_t = df +
\sum\limits_{j=1} ^\infty t^j \omega_j$, i.e., we have a deformation
$\omega_t$  of an exact 1-form $\omega_0=df$.

More precisely, we will consider the following situation:
$\{\omega_t\}_{ t\in (\mathbb C,0)}$ is an analytic deformation of
$\omega_0 = df$ such that each 1-form $\omega_t\in \Omega^1 (\mathbb
C^n,0)$ is integrable, $\omega_t \wedge d \omega_t=0$.

Let $(X_f,0)\subset (\mathbb C^{n},0)$ be a germ of reduced analytic
hypersurface as above. If $n=3$ and $(X_f,0)$ only has normal
crossings singularities off the origin $0 \in \mathbb C^{3}$, then
the local fundamental group  of the complement of $(X_f,0)$ in
$(\mathbb C^{n},0)$ is abelian. We have the following  general
statement below:

\begin{Theorem}[ L\^e-Saito, \cite{Le-Saito} Main Theorem page 1]
\label{Theorem:LeSaito} Let $n \geq 3$. Assume that outside of an
analytic subset $(Y,0)\subset (X_f,0)$ of dimension at most $n-3$,
the only singularities of $(X_f,0)$ are normal crossings. Then the
local fundamental group of the complement of $(X_f,0)$ in $(\mathbb
C^{n},0)$ is abelian. The Milnor fiber of $f$ has a fundamental
group which is free abelian of rank the number of analytic
components of $X_f$ at $0$, minus one. Finally, if $X_f$ is
irreducible, then the  fiber $f^{-1}(c), c \ne 0$  is
simply-connected.
\end{Theorem}

 Write now $f=f_1\ldots f_{r+1}$ in
terms of germs $f_j \in \mathcal O_n$ such that each irreducible
component of $X_f$ corresponds to one and only one of the sets
$(f_j=0)$. We shall consider logarithmic 1-forms $\theta_ \nu =
\sum\limits_{j=1}^{r+1} \lambda _j ^\nu df_j /f_j, \, \nu=1,...,r
\geq 1$ with the following property:
\begin{enumerate}[{\rm(P-1)}]

\item $\{ \theta_1,\ldots,\theta_r\}$ is completely independent with
respect to $df/f=\sum\limits_{j=1} ^{r+1} df_j /f_j$ in the
following sense:

if $\sum\limits_{\nu=1} ^r a_\nu \theta_\nu + b df/f=0$ for some
constants $a_\nu, b \in \mathbb C$ then $a_\nu=b=0$.

\end{enumerate}

This is the case if we have:

\[
\det\begin{pmatrix}
1 & \ldots & 1 \\\lambda_1^1 & \ldots & \lambda_{r+1} ^1\\
\vdots &  \dots & \vdots\\
\lambda_1 ^r & \ldots & \lambda_{r+1} ^1
\end{pmatrix}
\ne 0
\]

\begin{Lemma}
\label{Lemma:1homology} For each $ c \in\mathbb C\setminus \{0\}$
the 1-homology of the   local fiber $L_c : (f=c)\subset (\mathbb
C^n,0)$ is generated by the restrictions $\theta_j \big|_{L_c}, \,
j=1,...,r$.

\end{Lemma}

\begin{proof}
For simplicity we suppose $r=1$ and write $\theta:=
\theta_1=\lambda_1 df_1/ f_1 + \lambda_ 2 df_2 /f_2$. We have
$\lambda _1 \ne \lambda_2$. Given a fiber $L_c: (f=c) , c \ne 0$ we
have $f_1 f_2 =c$ on $L_c$. Thus $\frac{df_1}{f_1} \big|_{L_c} =
-\frac{df_2}{f_2}\big|_{L_c}$ and therefore $\theta\big|_{L_c} =
(\lambda _1 - \lambda_2 ) \frac{df_1}{f_1} \big|_{L_c}$. Since
$f=f_1 f_2=0$ is normal crossings outside of an analytic subset of
codimension $\geq 3$, at a generic point $p\in (f_1=f_2=0)$, $f_1$
and $f_2$ are part of a local system of coordinates $(x_1,...x_n)$
say $f_1=x_1, \, f_2 = x_2$. So you can consider a cycle
$\gamma_c\subset (f=c)$ of the following type $\gamma_c(s)=(\epsilon
e^{is}, \epsilon e^{-is}, 0,...,0), \epsilon^2=c, i^2=-1$.

 We claim that this is a non-trivial cycle in the homology of $L_c$.
Indeed,
\[
\int_{\gamma_c} \theta \big|_{L_c} = \int (\lambda_1 - \lambda_2)
df_1 / f_1 \big|_{L_c} = (\lambda_1 - \lambda_2 )  2 \pi \sqrt{-1}\,
k(\gamma_c)
\]
where $ k(\gamma_c) \ne 0$ is the index of $\gamma_c\subset \Sigma$
around the origin $p_1 \in \Sigma$. Thus $\int_{\gamma_c} \theta_c
\ne 0$.

Now, by Lê-Saito's theorem we have that $\pi_1(L_c) \cong \mathbb
Z$. In particular $H_1(L_c,\mathbb Z) \cong \pi_1(L_c)$ so that
$H_1(L_c,\mathbb Z) \cong \mathbb Z$ is free abelian of rank one.
Since $\theta$ is closed its restrictions to the fibers $L_c$ are
also closed and holomorphic. Therefore, $\theta\big|_{L_c}$
generates the group $H^1(L_c, \mathbb C)$ for each $ c \ne 0$. The
same argumentation works for the case $r > 1$.
\end{proof}

\begin{Remark}
For most  applications we shall take $\theta_j=\frac{df_j}{f_j}, \,
j=1,...,r$.
\end{Remark}

The above lemma then shows that the homology of the fibers $L_c, c
\ne 0$ is generated by restrictions of a same system of forms to
these fibers.

\begin{Proposition}
\label{Proposition:relativecohomologylocal} Let $\omega_1$ be a germ
of a holomorphic 1-form at $0 \in \mathbb C^n, \,n \geq 3$ and
assume that $d \omega_1 \wedge df=0$ where $f=f_1\ldots f_{r+1}$ is
as in Lemma~\ref{Lemma:1homology} above. Then there are $a_1, h_1
\in \mathcal O_n, \psi _j \in \mathcal O_1, \psi_j(0)=1$ and
$\lambda_j \in \mathbb C, j=1,...,r$ such that
\[
\omega_1 = a_1 df + dh_1 + \sum\limits_{j=1}^r \lambda _j f
\psi_j(f) \theta _j
\]

\end{Proposition}

\begin{proof}
We first consider the case $r=1$, i.e., $f=f_1f_2$. Given a generic
point $p_1 \in (f=0) \setminus \{0\}$ we may parametrize a
transverse disc $\Sigma$ to $df$, centered at $p_1$, by
$c=f\big|_{\Sigma}$. We define now a function $h_1 \colon W
\setminus \{p_1\} \to \mathbb C$ in some neighborhood $W\subset
\Sigma$ of $p_1$ as follows: let $\alpha(c) \in \mathbb C$ be
defined for each $c \ne 0$ by $[\omega_1\big|_{L_c}] = \alpha(c) [
\theta_1\big|_{L_c}]$ in $H^1(L_c,\mathbb C)$. Then we can define a
function $h_1$ in a neighborhood of the origin minus the
hypersurface $X_f \colon (f=0)$ say, $h_1 \colon U \setminus X_f \to
\mathbb C$ by setting $h_1(c)=c, \forall c \in \Sigma \setminus
\{0\}$ and, given $ z \in L_c$ we put $h_1(z)= \int\limits_{c}^{z}
\omega_1 \big|_{L_c} - \alpha (c) \theta_1 \big|_{L_c} =
\int_{\gamma_{(c,z)}} \omega_1 - \alpha(c) \theta _1$ where
$\gamma_{(c,z)}\subset L_c$ is any smooth path joining
$\gamma_{(c,z)}(0)=c$ to $\gamma_{(c,z)}(1)=z$ in $L_c$. Then we
have $dh_1\big|_{L_c} = (\omega_1 - \alpha(c) \theta_1)\big|_{L_c}$
and therefore
\[
dh_1 \wedge df = (\omega_1 - \alpha_1(f) \theta_1)\wedge df
\]
in $U\setminus X_f$. This implies $\omega_1 - \alpha_1(f) \theta_1 =
a_1 df + dh_1$ for some function $a_1 \colon U \setminus X_f \to
\mathbb C$ which is holomorphic. From $\omega_1 = a_1 df + dh_1 +
\alpha_1(f) \theta_1$ we have that
\[
\alpha_1 (c) = \dfrac{\int_{\gamma_c} \omega_1
\big|_{L_c}}{\int_{\gamma_c} \theta_1 \big|_{L_c}} = \frac{const.}{2
\pi \sqrt{-1}} \int _{\gamma_c} \omega_1 \big|_{L_c}, \forall c \ne
0,
\]
where $\gamma_c\subset L_c$ is a generator of the homology of the
fibers $L_c:  (f=c),\, \,  c \ne 0$. Since $\omega_1$ is holomorphic
in $U$ this implies that $\alpha_1(c)$ is holomorphic and bounded
for $ c \ne 0$. By Riemann's extension theorem, $\alpha_1(c)$ admits
a holomorphic extension to $c=0$. Now from
\[
(- \alpha_1(f)\theta_1 + \omega_1)\wedge df = dh_1 \wedge df
\]
we have that $h_1$ admits a holomorphic extension to $(f=0)$ as in
\cite{Ce-Sc} (final part of the proof of Proposition~5.2). Finally,
this implies that $a_1$ admits a holomorphic extension to $X_f$.
Since $\theta_1$ has poles of order $1$ on $X_f$ and $\omega_1,
a_1df$ and $dh_1$ are holomorphic, we conclude that $f \big|
\alpha_1(f)$, i.e., $\alpha_1(f)= f \psi_1(f)$ for some holomorphic
function $\psi_1(t) \in \mathcal O_1$.

Again, for the case $r>1$ there are no major changes. Indeed, using
the homology description of the fibers $(f=c), \, c \ne 0$ we let
$\alpha_j(c) \in \mathbb C, \, j=1,...,r$ be defined for each $c \ne
0$ by $[\omega_1\big|_{L_c}] = \sum\limits_{j=1}^r\alpha_j(c) [
\theta_j\big|_{L_c}]$ in $H^1(L_c,\mathbb C)$. Then, as before,  we
define a function $h_1 \colon U \setminus X_f \to \mathbb C$ by
setting $h_1(c)=c, \forall c \in \Sigma \setminus \{0\}$ and, given
$ z \in L_c$ we put $h_1(z)= \int\limits_{c}^{z} \omega_1
\big|_{L_c} - \sum\limits_{j=1}^r\alpha_j (c) \theta_j \big|_{L_c} =
\int_{\gamma_{(c,z)}} \omega_1 - \sum\limits_{j=1}\alpha_j(c) \theta
_j$ where $\gamma_{(c,z)}\subset L_c$ is any smooth path joining
$\gamma_{(c,z)}(0)=c$ to $\gamma_{(c,z)}(1)=z$ in $L_c$. Then we
have $dh_1\big|_{L_c} = (\omega_1 - \sum\limits_{j=1}^r \alpha_j(c)
\theta_j)\big|_{L_c}$ and therefore
\[
dh_1 \wedge df = (\omega_1 - \sum\limits_{j=1}^r\alpha_j(f)
\theta_j)\wedge df
\]
in $U\setminus X_f$. As before for the case $r=1$ this implies
$\omega_1 - \sum\limits_{j=1}^r\alpha_j(f) \theta_j = a_1 df + dh_1$
for some function $a_1 \colon U \setminus X_f \to \mathbb C$ which
is holomorphic.

Let now $\gamma^i _ c\subset L_c, \, i=1,...,r$ be a system of
generators of the homology of the fibers $L_c: (f=c),\,  c \ne 0$,
with the property that $\oint_{\gamma_c ^i}\theta_j= \delta_{ij}$
for the
 Kronecker delta $\delta_{ij}$.

From $\omega_1 = a_1 df + dh_1 + \sum\limits_{j=1}^r\alpha_j(f)
\theta_j$ we have that
\[
\alpha_j (c) = \dfrac{\int_{\gamma^j _c} \omega_1
\big|_{L_c}}{\int_{\gamma^j _c} \theta_j \big|_{L_c}} =
\frac{const.}{2 \pi \sqrt{-1}} \int _{\gamma^j_c} \omega_1
\big|_{L_c}, \forall c \ne 0,
\]
where $\gamma_c ^j \subset L_c$ is one of the above system of
generators of the homology of the fibers $L_c (f=c), c \ne 0$. Since
$\omega_1$ is holomorphic in $U$ this implies that $\alpha_j(c)$ is
holomorphic and bounded for $ c \ne 0$. By Riemann's extension
theorem, $\alpha_j(c)$ admits a holomorphic extension to $c=0$. Now
from
\[
(- \sum\limits_{j=1}^r\alpha_j(f)\theta_j + \omega_1)\wedge df =
dh_1 \wedge df
\]
we have that $h_1$ admits a holomorphic extension to $(f=0)$ as in
\cite{Ce-Sc}  (final part of the proof of Proposition~5.2). Finally,
this implies that $a_1$ admits a holomorphic extension to $X_f$.
Since $\theta_j$ has polar set of order $1$ and contained  in $X_f$
and $\omega_1, a_1df$ and $dh_1$ are holomorphic, we conclude that
$f \big| \alpha_j(f)$, i.e., $\alpha_j(f)= f \psi_j(f)$ for some
holomorphic function $\psi_j(t) \in \mathcal O_1$. This proves
Proposition~\ref{Proposition:relativecohomologylocal}.

\end{proof}

%%%%%%%%%%%%%%%%%%%%%%%%%%%%%%%%%%%%%%%%%%%%%%%%%%

\section{Relative Cohomology:  polynomial  case}
\label{section:homogeneous} We shall consider the polynomial case
for the cohomological equation $d \omega_1 \wedge df=0$ where
$f=f_1\ldots f_{r+1}$ is as  in Theorem~\ref{Theorem:A}.

\begin{Proposition}
\label{Proposition:omega1df}
 Given $f=f_1...f_{r+1}$ a product of relatively prime
 irreducible  homogeneous polynomials $f_j
 \in \mathbb C[x_1,...,x_n]$ of degree $\deg(f)=\nu+1$.
 Assume also that the induced germ $f \in \mathcal O_n$ has
only normal crossings singularities off a codimension $\geq 3$
subset. Let $\theta_1,...,\theta_r$ be a set of generators of the
1-homology of the fibers $L_c: (f=c)$ as in
Lemma~\ref{Lemma:1homology} (see also the proof of
Proposition~\ref{Proposition:relativecohomologylocal}). Given
$\omega_1$ a  polynomial 1-form of degree $\deg(\omega_1) \leq \nu$,
satisfying $d \omega_1 \wedge df=0$. Then we have
\[
\omega_1 = a_0 df + dh +  f \sum\limits_{j=1}^r \lambda_j  \theta_j
\]
for some constants $a_0, \lambda_j\in \mathbb C$ and some polynomial
$h$ of degree $\leq \nu +1$. The polynomial $h$ is homogeneous of
degree $\nu+1$ if and only if $\omega_1$ is homogeneous of degree
$\nu$.

\end{Proposition}

\begin{proof}
We first write, according to
Proposition~\ref{Proposition:relativecohomologylocal},  $\omega_1
=adf + dh + \sum\limits_{k=1}^{r} \lambda_k f \psi_k(f) \theta_k$
for some holomorphic functions $a, h \in \mathcal O_n, \psi_k \in
\mathcal O_1$ with $\psi_k(0)=1$ and constants $\lambda_k \in
\mathbb C$.

Now we consider expansions of $a, h$ and $\psi_k$ in sums of
homogeneous polynomials. Since $df$ and $f \theta_k$ are homogeneous
of same degree $\nu$ and since $\omega_1$ has degree $\leq \nu$, the
result follows immediately.

\end{proof}

 Now we make a couple of simple remarks:
\begin{enumerate}
\item  Let $h\in \mathbb C[x_1,...,x_n]_{\nu +1}$ and
$f=f_1\ldots f_{r+1}\in \mathbb C[x_1,...,x_n]_{\nu +1}$ where
$f_i,f_j$ are pairwise relatively prime, irreducible and reduced.
Then we have $f \big| dh \wedge df \Leftrightarrow h = \lambda f$
for some $\lambda \in \mathbb C$.

\begin{proof}
For the nontrivial part we consider the foliation $\fa_h$ given by
the hypersurfaces $h=const$. The leaves are the connected components
of $\{h=c\}\setminus \{0\}$ near the origin $0 \in \mathbb C^n$. The
fact that $f$ divides $dh\wedge df$ means that $(f=0)$ is invariant
by $\fa_h$ and therefore, since $0 \in (f=0)\cap (h=0)$ we conclude
that $(f=0)$ is contained in the level $(h=0)$ of $h$. This implies
that $h=pf$ for some polynomial $p$ homogeneous. Because of the
degrees we have $p=const$.

\end{proof}

\item If $(f=0)$ is invariant by $\omega_t=\sum\limits_{j=0}^\infty t^j \omega_j$
for all $t$ then $f \big| \omega_j \wedge df$ for all $j$.

\begin{proof}
Since $(f=0)$ is invariant by $\omega_t$ for all $t$ we have that
$f$ divides $\omega_t \wedge df$ for all $t$. Since $\omega_t \wedge
df = \sum\limits_{j=0}^\infty t^j \omega_j \wedge df$ in power
series in $t$ we conclude.

\end{proof}

\end{enumerate}

%%%%%%%%%%%%%%%%%%%%%%%%%%%%%%%%%%%%%%%%%%%%%%%%%%%%

%%%%%%%%%%%%%%%%%%%%%%%%%%%%%%%%%%%%%%%%%%%%%%%%%%%%
\section{Integration lemma: proof of Theorems~\ref{Theorem:A} and~\ref{Theorem:B}}
\label{section:integration} This section is dedicated to the proof
of Theorems~\ref{Theorem:A} and ~\ref{Theorem:B}.
\begin{proof}[Proof of Theorem~\ref{Theorem:A}]
We shall then consider deformations $\omega_t= df +
\sum\limits_{j=1}^\infty t^j \omega_j$ where $f=f_1...f_{r+1}$ is a
product of homogeneous polynomials $f_j$ and has normal crossings
off a codimension $\geq 3$ subset. We assume that each $\omega_t$ is
integrable and polynomial of degree $\deg (\omega_t) \leq
\deg(df)=\nu$. The forms $\omega_t$ though polynomial are not
necessarily homogeneous. We shall first prove (c)$\implies$(b) in
Theorem~\ref{Theorem:A}.

The first deformation equation $d\omega_1 \wedge df=0$ gives by
Proposition~\ref{Proposition:omega1df}
\[
\omega_1 = a df + dh + f\sum\limits_{j=1}^r \lambda_j  \theta_j
\]
for some constants $a, \lambda_j\in \mathbb C$ and some polynomial
$h$ of degree $\leq \nu +1$.  To simplify we proceed with the
assumption that $r=1$. In this simplified notation we have
$f=f_1f_2$ and we may write $\omega_1 = a_1 df + dh_1 + f \lambda_1
\theta$ for some $a_1, \lambda_1 \in \mathbb C$ and some polynomial
$h_1\in \mathbb C[x_1,...,x_n]$ of degree $\deg(h_1) \leq \deg(f)$.
Hence we can write $\omega_t = df + t(a_1 df + dh_1 + \lambda_1 f
\theta) + t^2 \omega_2 + \ldots$. Assuming now that
$\int_{\gamma_c}\omega_t=0, \forall t$ we conclude that $\lambda_1
=0$ and $\int _{\gamma_c}\omega_j=0, \forall j \geq 2$. Since
$\omega_1$ is closed the second deformation equation writes
$d\omega_2 \wedge df =0$. Also we have  $\omega_t = (1 + ta_1) df +
t dh_1 + t^2 \omega_2 + \ldots$. From $d\omega_2 \wedge df =0$,
again  via Proposition~\ref{Proposition:omega1df},  we obtain
$\omega_2 =a_2 df + dh_2 + \lambda_2 f \theta$ for some constants
$a_2, \lambda _2\in \mathbb C$ and $h_2 \in \mathbb C[x_1,...,x_n]$
polynomial of degree $\leq \nu+1$. Then since $\int _{\gamma_c}
\omega_2 =0$ we have $\lambda_2=0$ and so on. As above for the
homogeneous case we obtain a formal function

\[
\hat F = f + (\sum\limits_{j=1}^\infty t^j a_j) f + \sum
\limits_{j=1}^\infty t^j h_j
\]
such that $\omega_t =d_x \hat F$; where each $h_j$ is a polynomial
of degree $\leq \nu+1$. We can therefore write $\hat F= (1+ \hat
A(t)) f + \hat H$ with $\hat H \in \mathbb C[x_1,...,x_n]_{\nu
+1}\bigotimes \hat {\mathcal O}_1$ and since $\omega_t= d_x\hat F$
converges we conclude that $\hat F(x,t)$ converges, i.e., $\hat F
\in \mathcal O_{n+1}$, indeed, $\hat F \in \mathbb
C[x_1,...,x_n]_{\nu +1} \bigotimes \mathcal O_1$. This shows that
$\omega_t$ is exact, admits a polynomial first integral of degree
$\leq \nu+1$. The case $r>1$ is similar to this. Now we observe that
(b)$\implies$(a). Also, (c)$\Leftrightarrow$(d) obviously. It is
also clear that (a)$\implies$(c).  Theorem~\ref{Theorem:A} is now
proved.
\end{proof}

\vglue.1in

The next example shows the necessity of the integral condition in
Theorem~\ref{Theorem:A}. Indeed we have:

\begin{Example}
\label{Example:cycle}{\rm Let us consider $\omega_t= d(xy) + t(xdy -
\lambda ydx)$ where $\lambda \in \mathbb C$. Then $\omega_t$ is
integrable and $\omega_0=d(xy)$. Put $f=xy$ then for $c \ne 0$ we
have the fiber $L_c: (f=c)\subset \mathbb C^2$ given by $xy=c$ and
admitting a 1-homology generator $\gamma_c(s)=(x e^{is}, ce^{-is}),
\, 0 \leq s \leq 2 \pi$. Then we have
\[
\int_{\gamma_c}\omega_t /f = t \int_{\gamma_c} \frac{(xdy - \lambda
ydx)} {xy}\big|_{xy=c} =t \int\limits_{0}^{2 \pi} (\frac{dy}{y} -
\lambda \frac{dx}{x})\big|_{\gamma_c}
\]
\[
 = t \int\limits_{0}^1{2\pi}  ( - i - \lambda i) ds = -  i (1+
\lambda) 2 \pi.
\]
Then $\int_{\gamma_c} \omega_t /f = 0 \Leftrightarrow \lambda =
-1\Leftrightarrow \omega_t = (1+ t) d(xy)$. Notice that if $\lambda
\ne -1$ then $\omega_t$ does not admit a holomorphic first integral.
On the other hand, $\omega_t$ is homogeneous of a same degree for
all $t$.

} \end{Example}

\begin{proof}[Proof of Corollary~\ref{Corollary:1}]
This is a direct consequence of Theorem~\ref{Theorem:A} once the
irreducibility of $P$ implies (by Lê-Saito's theorem) that the
non-singular fibers $(P=c), c \ne 0$ are simply-connected. On the
other hand, since $P$ is homogeneous, the existence of local
solutions for the cohomology equation $d\omega_1 \wedge dP=0$ of the
form $\omega_1=a_1 dP + dh_1$ with $a_1, h\in \mathcal O_n$ implies,
as in the above given proof of Theorem~\ref{Theorem:A},  for
$\omega_1$ polynomial 1-form of degree $\leq \nu$, the existence of
solutions with $a_1=const$ and $h_1\in \mathbb C[x_1,...,x_n]$ of
degree $\leq \nu+1$. We then proceed as in the proof of
Theorem~\ref{Theorem:A}.

\end{proof}

\begin{proof}[Proof of Theorem~\ref{Theorem:B}]
We write $\Omega= dP + \sum\limits_{j=\nu+1} ^\infty \Omega_j$ where
$\Omega_j$ is a homogeneous 1-form of degree $j \geq \nu+1$. Then,
we consider the maps $\sigma _t \colon \mathbb C^n \to \mathbb C^n$
given by $\sigma_t(z)=tz, \, t \in \mathbb C, \, z \in \mathbb C^n$.
 We then put
$\omega_t:=\frac{1}{t^{\nu+1}} \sigma_t^*(\Omega)$. Because $\Omega
\wedge d\Omega=0$ we also have $\omega_t \wedge d \omega_t=0$ for
each $t \in \mathbb C$. Notice that for each $t \ne 0$ we have
$\sigma_t$ an automorphism of $\mathbb C^n$ and also $\sigma_t ^*
(\Omega) = t^{\nu+1} dP + \sum\limits_{j=\nu+ 1}^\infty \sigma^* _t
(\Omega_j) = t^{\nu+1}[dP + t\sum\limits_{j=\nu+1} ^\infty t^{j+1 -
(\nu +2)} \Omega_j]$.
\begin{Claim}
The family of 1-forms $\omega_t$ defines an analytic deformation of
$\omega_0=dP_{\nu+1}$ by integrable 1-forms, such that
$\omega_{t=1}=\Omega$. If $\Omega$ is polynomial  of degree $\mu$
then the same holds for each $\omega_t$.
\end{Claim}
Thanks to the main result in \cite{Ce-Sc} we conclude that
$\omega_t$ has a holomorphic first integral for $t \approx 0$ in a
neighborhood of the origin $0 \in \mathbb C^n$. Indeed, there is a
holomorphic function $F(z,t)$ defined in a neighborhood of $(0,0)\in
\mathbb C^n \times \mathbb C$ such that:
\begin{enumerate}[{\rm(I)}]
\item $d_zF(z,t)\wedge \omega_t=0$
\item $F(z,0)=P(z)$.
\end{enumerate}
Now we observe that $\sigma_t\in \Aut(\mathbb C^n,0), \forall t \ne
0$. Therefore we conclude the existence of $F(z,t)$ for all $t \in
\mathbb C$ and in particular for $t=1$. Since $\sigma_1(z)=z$ we get
a holomorphic first integral for $\omega_{t=1} = \Omega$ in a
neighborhood of the origin.
\end{proof}

\section{Degree one deformations: proof of Theorem~\ref{Theorem:C}}
\label{section:degreeone} We now turn our attention to the affine
deformation case $\omega_t=df + t \omega_1$ when $\omega_1$ is
polynomial of degree $\leq \deg(f)$ where $f$ is homogeneous as in
Theorem~\ref{Theorem:A}. Given $f=f_1...f_{r+1}$ a product of
relatively prime
 irreducible  homogeneous polynomials $f_j
 \in \mathbb C[x_1,...,x_n]$ of degree $\deg(f)=\nu+1$.
 Assume also that the induced germ $f \in \mathcal O_n$ has
only normal crossings singularities off a codimension $\geq 3$
subset. Let $\theta_1,...,\theta_r$ be a set of generators of the
1-homology of the fibers $L_c: (f=c)$ as in
Lemma~\ref{Lemma:1homology} (see also the proof of
Proposition~\ref{Proposition:relativecohomologylocal}).

For the case where $ r \geq 2$ we may choose the coefficients of
$\theta_j, j=1,...,r$,  in such a way that if we have a holomorphic
function $h\in \mathcal O_n$ such that $dh \wedge \theta _j =0$ for
some $j$ then $h$ is constant. Nevertheless, this does not matter in
the following proof where we consider $r=1$.

\begin{proof}[Proof of Theorem~\ref{Theorem:C}]
We shall now prove Theorem~\ref{Theorem:C}. We start with the more
general framework, ie., with $f=f_1\ldots f_{r+1}$ because we want
to make some considerations about this case also.

Recall that the 1-forms $\omega_t$ are not necessarily homogeneous.
From the first deformation equation $d\omega_1 \wedge df=0$ where
$f=f_1\ldots f_{r+1}$ we get via
Proposition~\ref{Proposition:omega1df}
\[
\omega_1 = a_1 df + dh + \sum\limits_{j=1} ^{r} \lambda_j f\theta_j
\]
for some constants $a_0,\lambda_j \in \mathbb C$ and some polynomial
$h$ of degree $\deg(h) \leq \nu+1$. Since we may take
$\theta_j=\frac{df_j}{f_j}$ we obtain

\[
\omega_1 = a_1 df + dh + \sum\limits_{j=1} ^{r} \lambda_j f
\frac{df_j}{f_j}
\]

Let us now focus on  the  case $r=1$, i.e., $f=f_1 f_2$. Write
$\omega_1 = a_1 df + dh + \lambda f \theta$ with $a_1, \lambda \in
\mathbb C$, $\deg(h) \leq \nu+1$ and $\theta= \theta_1 = df_1/f_1$.
The second deformation equation $\omega_1 \wedge d \omega_1 =0$
gives
\[
(a_1 df + dh + \lambda f \theta)\wedge (\lambda df \wedge \theta)=0.
\]
From this last equation we obtain the equivalent equation
\[
\lambda dh \wedge df \wedge \theta =0.
\]
Let us investigate the solutions to this last equation.

$\bullet$ If $\lambda=0$ then $\omega_1 = a_1 df + dh$ and $\omega_t
= df + td(a_1f + h)= df + td \tilde h= d(f + t\tilde h)$ as in (a)
in Theorem~\ref{Theorem:C}.

$\bullet$ Assume now that $\lambda \ne 0$. In this case we have
\[
d(h\theta) \wedge df=dh \wedge df \wedge \theta=0.
\]
Notice that
\[
\omega_t = df + t ( a_1 df + dh + \lambda f \theta) = (1 + a_1t) df
+ t dh + t \lambda f \theta
\]
From $d(h\theta) \wedge df=0$ we have $dh \wedge
\frac{df_1}{f_1}\wedge \frac{df_2}{f_2}=0$. Therefore
\[
dh \wedge df_1 \wedge df_2=0.
\]

Since $(f_1,f_2)$  satisfies the factorization property, $dh \wedge
df_1 \wedge df_2=0$ implies that $h=h_1(f_1,f_2)$ for some
polynomial $h_1 \in \mathbb C[x,y]$ in two variables with
$h_1(0,0)=0$. Let $\nu_j=\deg (f_j), \, j=1,2$. Then $f=f_1f_2$ has
degree $\nu+1=\nu_1 + \nu_2$\,. Since $\deg (h) \leq \nu+1$ we must
have $h_1(x,y)=\sum\limits_{i=1}^ {\nu_2} a_ i x ^i +
\sum\limits_{j=1} ^{\nu_1} b_j y^j + c xy$ for some constants $a_i,
b_j, c \in \mathbb C$. Let now $\sigma \colon \mathbb C^n \to
\mathbb C^2$ given by $\sigma(z)= (f_1(z),f_2(z))$. Then we can
write $\omega_t = \sigma^*(\alpha_t)$ where $\alpha_t$ is the
one-parameter family of two variables 1-forms given by
\[
\alpha_t = d(xy) + t ( a_1 d(xy) + dh_1(x,y) + \lambda y dx) = (1 +
t \mu)d(xy) + t d(\sum\limits_{i=1}^ {\nu_2} a_ i x ^i +
\sum\limits_{j=1} ^{\nu_1} b_j y^j) + t\lambda y dx
\]
\[
=(1 + t \mu)d(xy) + td(P(x) + Q(y))  + t\lambda y dx
\]
for some $\lambda, \, \mu \in \mathbb C$ and polynomials $P(x)\in
\mathbb C[x], \, Q(y) \in \mathbb C[y]$ of degree $\deg P \leq
\nu_2, \, \deg Q \leq \nu_1$.

\end{proof}

\begin{Remark}
{\rm  If we consider  degree one deformations as in
Theorem~\ref{Theorem:C} but with $f$ of the form
$f=f_1f_2f_3...f_{r+1}$ with $r \geq 2$ then we have to study the
solutions of the equation
\[
dh \wedge df \wedge \sum\limits_{j=1}^r \lambda_j \frac{df_j}{f_j}=0
\]
as it comes from the proof of Theorem~\ref{Theorem:C}. }
\end{Remark}

\bibliographystyle{amsalpha}

\begin{thebibliography}{31}
\frenchspacing


\bibitem[Camacho-Lins Neto 1985]{C-LN} C. Camacho, A. Lins Neto;
{\em Geometric theory of foliations}, Translated from the Portuguese
by Sue E. Goodman. Birkhauser Boston, Inc., Boston, MA (1985).


\bibitem[Camacho-Lins Neto-Sad 1984]{C-LN-S1} C. Camacho, A. Lins Neto
and P. Sad: {\it   Topological invariants and
equidesingularization for holomorphic vector fields}; J. of Diff.
Geometry, vol. 20, no.\ 1 (1984), 143--174.



\bibitem[Cerveau-Berthier 1993]{cerveau-berthier}
Michel Berthier; Dominique Cerveau: {\it Quelques calculs de
cohomologie relative}. Annales scientifiques de l'École Normale
Supérieure (1993) Volume: 26, Issue: 3, page 403-424

\bibitem[Cerveau-Scárdua 2018]{Ce-Sc} D. Cerveau, B. Scárdua:
{\it Integrable deformations of local analytic fibrations with
singularities}, Arkiv för Matematik,  Volume 56 (2018) Number 1,
Pages: 33 - 44 DOI: http://dx.doi.org/10.4310/ARKIV.2018.v56.n1.a3.



\bibitem[Cerveau-Mattei 1982]{Cerveau-Mattei} D. Cerveau, J.-F. Mattei:
Formes intégrables holomorphes singulières, Astérisque, Vol.
97 (1982). MR 86f:58006

\bibitem[Deligne 1979]{deligne} P. Deligne: {\it Le groupe
fundamental du compl\'ement
d'une courbe plane n'ayant que des points doubles ordinaires est
ab\'elien}. S\'eminaire Bourbaki, Vol. 79/90 nov. 1979.

%\bibitem[Grauert-Remmert 1979]
%{Gr-Re} H. Grauert, R. Remmert,  \, Theory of Stein Spaces, Springer 1979.


\bibitem[Gunning 1990]{gunning1} R.C. Gunning: { Introduction to
holomorphic functions of several variables}, vol. I, Function
Theory. Wadsworth \& Brooks/Cole Advanced Books \& Software, Pacific
Grove, CA, 1990.

\bibitem[Gunning 1992]{gunning2} R.C. Gunning: { Introduction to
holomorphic functions of several variables}, vol. II, Local Theory.
Wadsworth \& Brooks/Cole Advanced Books \& Software, Monterey, CA,
199



\bibitem[Gunning-Rossi 1965]{gunning-rossi} R. Gunning \& H. Rossi:
Analytic functions of several complex
variables, Prentice Hall, Englewood Cliffs, NJ, 1965.



\bibitem[Ilyashenko 1969]{Ilyashenko} Ilyashenko S., The origin of
limit cycles under perturbation of the equation $dw/dz = - Rz /Rw$ ,
where  $R(z, w)$ is a polynomial. 1969 Math. USSR Sb. 7 353.

\bibitem[Kaup-Kaup 1983]{kaup-kaup} Kaup, Ludger / Kaup, Burchard:
Holomorphic Functions of Several Variables An Introduction to the
Fundamental Theory In coop. with Barthel, Gottfried Transl. by
Bridgland, Michael / Bridgland, Michael Series:De Gruyter Studies in
Mathematics 3, 1983.


\bibitem[Lê-Saito 1984]{Le-Saito} Lê Dung Tráng, Kyoji Saito:
{\it The local $\pi_1$ of the complement of a hypersurface
with normal crossings in codimension $1$ is abelian}, Ark. Math. 22
(1984), no. 1, 1-24. MR 735874 (36a:32019)



\bibitem[Malgrange 1976] {malgrangeI}
B. Malgrange: {\it Frobenius avec singularités, 1. Codimension un},
Public. Sc. I.H.E.S., 46 (1976), pp. 163-173.

\bibitem[Malgrange 1977] {malgrangeII} B. Malgrange:
{\it  Frobenius avec singularit\'es. 2. Le cas general}. Inventiones
mathematicae (1977) Volume: 39, page 67-90.

\bibitem[Mattei-Moussu 1980] {mattei-moussu}
J.F. Mattei and R. Moussu:
{\it  Holonomie et int\'egrales premi\`eres},
Ann. Sci. \'Ecole Norm. Sup. (4) {\bf 13} (1980), 469--523.


\bibitem[Milnor 1968]{Milnor} J. Milnor: Singular points of complex hypersurfaces. Ann.
Math. Studies, Vol. 61, Princeton Univ. Press, 1968.


\bibitem[Movasati 2004]{movasati} Movasati, H., Center conditions:
rigidity of logarithmic differential equations, J. Diff. Equ. 197
(2004) 197-217.
\bibitem[Reeb 1952]{Reeb} G. Reeb; {\it Sur certaines propri\'et\'es
topologiques des vari\'et\'es feuillet\'ees};
Actualit\'es Sci. Ind., Hermann, Paris, 1952.



\bibitem[Saito 1976]{saito} Kyoji Saito: {\it On a generalization of de Rham lemma.}
Annales de l'institut Fourier (1976), Volume: 26, Issue: 2, page
165-170,

\end{thebibliography}

\end{document}